\documentclass[letterpaper,12pt,leqno]{scrartcl}
\usepackage{amsmath,upref,amssymb,amsthm,amsxtra,exscale}
\usepackage{cite}
\usepackage{enumitem}

\numberwithin{equation}{section}

\theoremstyle{plain}
\newtheorem{theorem}{Theorem}[section]
\newtheorem{lemma}[theorem]{Lemma}
\newtheorem{proposition}[theorem]{Proposition}

\theoremstyle{definition}
\newtheorem{remark}[theorem]{Remark}

\newtheorem{example}[theorem]{Example}

\newcommand{\dN}{\mathbb{N}}

\newcommand{\dR}{\mathbb{R}}

\newcommand{\cK}{\mathcal{K}}

\newcommand{\rmd}{\mathrm{d}}

\newcommand{\olB}{\overline{B}}

\DeclareMathOperator{\supp}{supp}

\DeclareMathOperator*{\esssup}{ess\,sup}

\newcommand{\dint}{\,\rmd}
\newcommand{\ssm}{\backslash}
\newcommand{\weakto}{\rightharpoonup}

\newcommand{\lr}[3]{#1#3#2}
\newcommand{\xlr}[3]{\left#1#3\right#2}
\newcommand{\biglr}[3]{\bigl#1#3\bigr#2}

\newcommand{\abs}[1]{\lr\lvert\rvert{#1}}
\newcommand{\xabs}[1]{\xlr\lvert\rvert{#1}}

\newcommand{\norm}[1]{\lr\lVert\rVert{#1}}

\newcommand{\scp}[1]{\lr\langle\rangle{#1}}

\newcommand{\coloneqq}{:=}

\def\rn{\mathbb{R}^N}

\def\eps{\varepsilon}
\def\rh{\rightharpoonup}

\def\io{\int_{\Omega}}
\def\vp{\varphi}
\def\wt{\widetilde}

\begin{document}

\title{A concentration phenomenon for semilinear
  elliptic equations}
 
\author{Nils Ackermann\thanks{Supported by CONACYT grant 129847
    and PAPIIT grant IN106612 (Mexico)} \and Andrzej
  Szulkin\addtocounter{footnote}{5}\thanks{Supported in part by
    the Swedish Research Council}}
\date{}
\maketitle
\begin{abstract}
  For a domain $\Omega\subset\dR^N$ we consider the equation $
  -\Delta u + V(x)u = Q_n(x)\abs{u}^{p-2}u$ with zero Dirichlet
  boundary conditions and $p\in(2,2^*)$.  Here $V\ge 0$ and $Q_n$
  are bounded functions that are positive in a region contained
  in $\Omega$ and negative outside, and such that the sets
  $\{Q_n>0\}$ shrink to a point $x_0\in\Omega$ as $n\to\infty$.
  We show that if $u_n$ is a nontrivial solution corresponding to
  $Q_n$, then the sequence $(u_n)$ concentrates at $x_0$ with
  respect to the $H^1$ and certain $L^q$-norms.  We also show
  that if the sets $\{Q_n>0\}$ shrink to two points and $u_n$ are
  ground state solutions, then they concentrate at one of these
  points.


  \end{abstract}
  
  \section{Introduction}
  \label{sec:introduction}

  Let $\Omega\subset \rn$ be a domain and consider the problem
  \begin{equation} \label{Eq} -\Delta u + V(x)u =
    Q(x)\abs{u}^{p-2}u, \quad u\in H^1_0(\Omega),
  \end{equation}
  where $H^1_0(\Omega)$ is the usual Sobolev space. Suppose
  $V,Q\in L^\infty(\Omega)$, $V\ge 0$ and $2<p<2^*$, where $2^*
  := 2N/(N-2)$ if $N\ge 3$, $2^* := \infty$ if $N=1$ or 2. If
  $\Omega$ is unbounded, assume in addition that $0$ is not in
  the spectrum of $-\Delta+V$ (i.e., $\sigma(-\Delta+V)\subset
  (0,\infty)$; this condition is automatically satisfied for
  bounded $\Omega$). Multiplying \eqref{Eq} by $u$ and
  integrating over $\Omega$ it follows immediately that $u=0$ is
  the only solution if $Q\le 0$. On the other hand, if $Q>0$ on a
  bounded set of positive measure, then it is easy to see that
  there exists a solution $u\ne 0$ to \eqref{Eq}. This will be
  shown in the next section and is in principle well known,
  cf.~\cite[Theorem~6]{MR96i:35033}. Assume without loss of
  generality that $0\in\Omega$ and let $Q=Q_n$ be such that
  $Q_n>0$ on the ball $B_{1/n}(0)$ and $Q_n<0$ on
  $\Omega\setminus B_{2/n}(0)$. For each $n$ there exists a
  solution $u_n\ne 0$, and in view of the discussion above it is
  natural to ask what happens with $u_n$ as $n\to\infty$. It is
  the purpose of this paper to show that the functions $u_n$
  concentrate at $x=0$. This concentration phenomenon does not
  seem to be earlier known.

  There is also another aspect of equation \eqref{Eq}, related to
  physics, or more specifically, to the propagation of
  electromagnetic waves which in our case is monochromatic light
  travelling through an optical cable (waveguide).  The transport
  of light in dielectric media is controlled by Maxwell's
  equations (ME) and an important role is played by the
  dielectric response $\varepsilon$ which may vary with location
  and light intensity, see e.g.~\cite{MR1079182}.  In the
  following denote by $\omega>0$ the frequency of light and by
  $c$ the speed of light in a vacuum, and put $\wt\eps =
  \frac{\omega^2}{c^2}\varepsilon$ for convenience.

  Our equation \eqref{Eq} is inspired by two models of optical
  waveguides~\cite[pp.~67--68]{Buryak200263}.  The first model
  concerns a stratified medium in $\dR^3$ consisting of slabs of
  dielectric materials that are perpendicular to the $x_1$-axis.
  Here we assume that the light beam is a wave travelling in the
  direction of $x_3$, having polarization in the direction of
  $x_2$, and $\wt\varepsilon$ is a function of $x_1$ and
  $\abs{u}^2$.  With the ansatz $E(x,t) = u(x_1)\cos(kx_3-\omega
  t)e_2$ for the electric field, where $e_2$ is the unit vector
  in the direction of $x_2$ and $k>0$ is the wave number, one
  obtains a guided solution of ME in the form of a plane
  travelling wave if and only if $u\in H^1(\dR)$ is a solution of
  the equation
  \begin{equation}
    \label{eq:12}
    -u''+\biglr(){k^2-\wt\varepsilon(x_1,\abs{u}^2)}u=0
    \qquad\text{in } \dR,
  \end{equation}
  see \cite{MR1245069,MR2361577} and the references there.  The
  total energy per unit length in $x_3$ of the wave is finite on
  each plane $\{x_2\equiv\mathrm{const.}\}$.  Note how the
  $x_1$-dependence of $\varepsilon$ exhibits the geometry of the
  waveguide.  We remark that here and in what follows there is no
  term $i\partial u/\partial x_3$ which appears in
  \cite{Buryak200263}. The reason is that unlike in
  \cite{Buryak200263} we always assume that $u$ is independent
  of~$x_3$.

  In the second model we assume
  $\wt\varepsilon=\wt\varepsilon(x_1,x_2,\abs{u}^2)$ and make the
  ansatz $E(x,t)= u(x_1,x_2)\cos(kx_3-\omega t)e_2$, the
  so-called \emph{scalar approximation} for a linearly polarized
  wave propagating in the $x_3$-direction.  Here one requires
  $u\in H^1(\dR^2)$ to be a solution of
  \begin{equation}
    \label{eq:14}
    -\Delta u+\biglr(){k^2-\wt\varepsilon(x_1,x_2,\abs{u}^2)}u=0
    \qquad\text{in } \dR^2.
  \end{equation}
  This ansatz does not yield solutions to ME, but it is
  nevertheless studied extensively in the relevant literature,
  cf.~\cite[p.\ 87]{Buryak200263}, \cite{MR519654,MR1079182} and
  the references given there.  In this case the total energy per
  unit length in $x_3$ of the wave is finite on $\dR^2$. One may
  also assume cylindrical symmetry, i.e., one puts $\wt\eps =
  \wt\eps(r,\abs{u}^2)$ and looks for solutions of the form
  $u=u(r)$, where $r^2=x_1^2+x_2^2$.

  In a nonlinear medium $\wt\varepsilon$ has a nontrivial
  dependence on $\abs{u}^2$.  The approximation
  \begin{equation*}
    \wt\varepsilon(x,\abs{u}^2) 
    = A(x)+
    Q(x)\abs{u}^{p-2}
  \end{equation*}
  is commonly used as long as $\abs{u}$ is not too large, so our
  equation \eqref{Eq} is the direct analogue of \eqref{eq:12} or
  \eqref{eq:14} in arbitrary dimension, with $V\coloneqq k^2-A$.
  This approximation is called the Kerr nonlinearity if $p=4$ and
  plays an important role in the physics
  literature~\cite{MR1245069}.  However, also $p\ne 4$ is of
  interest (non-Kerr-like materials), as are dielectric response
  functions corresponding to saturation (which occurs when
  $\abs{u}$ becomes large), see \cite[note added in
  proof]{MR1079182}, \cite{MR1079182} and the references
  there. In this latter case the response is of the form
  $A(x)+Q(x)g(\abs{u})$, with $g(0)=0$, $g$ increasing and
  $\lim_{\abs{u}\to\infty}g(\abs{u})$ finite. This leads to the
  right-hand side $Q(x)g(\abs{u})u$ in \eqref{Eq}. The part of
  the medium where $Q>0$ is called self-focusing (the dielectric
  response increases with $\abs{u}$) and the part where $Q<0$ is
  called defocusing. So if $Q>0$ on a set of small size, the
  medium has a self-focusing core and is defocusing outside of
  this core.

  It is common to consider materials separately with $Q$ positive
  or negative, see e.g. \cite[Eq.~(48)]{Buryak200263}, which
  corresponds to investigating the existence of bright ($Q>0$) or
  dark ($Q<0$) solitons.  However, also materials with
  sign-changing $Q$ are considered. In this vein, see
  \cite{RevModPhys.83.247}, or \cite[Eq.~(3)]{PhysRevA.83.033828}
  for an example where a sharp localization of the self-focusing
  region is considered. There is also recent evidence that
  materials with a large range of prescribed optical properties
  can be
  created~\cite{Veselago:1968,Pendry23062006,shalaev_optical_2007,Smith06082004}.
  Therefore it is reasonable to prescribe the nonlinear
  dielectric response almost at will for each material.

  The conditions we impose on the functions $Q_n$ allow to model
  a composite of two materials where the size of the
  self-focusing core decreases as $n\to\infty$.  In particular,
  we show for the plane travelling waves introduced above by way
  of \eqref{eq:12} and for the Kerr nonlinearity that the field
  $E$ concentrates on the $x_1$-axis in the sense of the $H^1$-
  and $L^q$-norms for all $q>1$ as $n\to\infty$, see
  Theorem~\ref{prop:conc-lq-norm} and Remark~\ref{rem:kerr}.
  Concerning the scalar approximation \eqref{eq:14} we obtain
  concentration at $(x_1,x_2)=(0,0)$ in $H^1$ and $L^q$ for every
  $q>2$ as $n\to\infty$ but not in the physically relevant case
  $q=2$. We do not know whether concentration in $L^2$ occurs
  here.

  There are numerous rigorous mathematical results on the effect
  of a sign changing $Q$ on the existence and properties of
  solutions of \eqref{Eq}. E.g. in \cite{MR2531172,MR2143511} $Q$
  takes the form $a_+-\mu a_-$ with $a_\pm\ge0$ continuous
  functions and $\mu\to\infty$.  A similar analysis for $Q=\delta
  a_+-a_-$ and $\delta\to0$ is contained in \cite{MR1615999}.
  Similarly as in our results the relative contribution of the
  negative and the positive part of $Q$ varies with a changing
  parameter.  Observe though that the change there occurs in the
  values of $Q$ while the regions where $Q>0$ and $Q<0$ are
  fixed.  The only result we are aware of that deals with
  changing the set $\{x: Q(x)>0\}$ through a parameter is
  \cite{MR1675283}.  In that paper a small region of diameter
  $\delta>0$ with $Q\equiv0$ is enclosed in a region where $Q>0$,
  and the behaviour as $\delta\to0$ is considered.  Nevertheless,
  this is different from our case, where a region with $Q<0$
  encloses a core with $Q>0$.

  \medskip

  Now we formulate our assumptions in a precise manner. Let
  $\Omega$ be a domain in $\rn$ and assume without loss of
  generality that $0\in\Omega$. $\Omega$ may be unbounded and we
  do not exclude the case $\Omega=\rn$. We will be concerned with
  the problem
  \begin{equation}
    \label{eq:21}\tag{$\mathrm{P}_n$}
    \left\{
      \begin{gathered}
        -\Delta u+V_n(x) u =Q_n(x)\abs{u}^{p-2}u, \quad x\in\Omega\\
        u(x)=0 \text{ as } x\in\partial\Omega, \quad u(x)\to 0
        \text{ as } \abs{x}\to\infty,
      \end{gathered}
    \right.
  \end{equation}
  where $p\in(2,2^*)$. Of course, the first condition in the
  second line of \eqref{eq:21} is void if $\Omega=\rn$ and the
  second condition is void if $\Omega$ is bounded.  We make the
  following assumptions concerning $V_n$ and $Q_n$:

  \begin{enumerate}[label=(A\arabic{*})]
  \item \label{item:3} $V\in L^\infty(\Omega)$, $V\ge 0$ and
    $\sigma(-\Delta+V)\subset(0,\infty)$, where the spectrum
    $\sigma$ is realized in $H^1_0(\Omega)$. $V_n = V+K_n$, where
    $K_n\in L^\infty(\Omega)$, and there exists a constant $B$
    such that $\norm{K_n}_\infty \le B$ for all $n$. Moreover,
    for each $\eps>0$ there is $N_\eps$ such that $\supp
    K_n\subset B_\eps(0)$ whenever $n\ge N_\eps$.
  \item \label{item:4} $Q_n\in L^\infty(\Omega)$, $Q_n>0$ on a
    set of positive measure and there exists a constant $C$ such
    that $\norm{Q_n}_\infty\le C$ for all $n$. Moreover, for each
    $\eps>0$ there exist constants $\delta_\eps>0$ and $N_\eps$
    such that $Q_n\le -\delta_\eps$ whenever $x\notin B_\eps(0)$
    and $n\ge N_\eps$.
  \end{enumerate}

  The following are two typical examples of $Q_n$ which we have
  in mind.
  \begin{example}
    \begin{enumerate}[label=\textup{(\alph{*})}]
    \item We let $\eps_n\to 0$ and
      \begin{equation*}
        Q_n(x) := \left\{
          \begin{array}{rl}
            1, & \abs{x}<\eps_n\\
            -1, & \abs{x}>\eps_n.
          \end{array} \right.
      \end{equation*}
    \item Let $Q$ be a bounded continuous function such that
      $Q(x)<Q(0)$ for all $x\ne 0$ and the diameter of the set
      $\{x: \lambda\le Q(x)\le Q(0)\}$ tends to 0 as
      $\lambda\nearrow Q(0)$. Put $Q_n(x):=Q(x)-\lambda_n$, where
      $\lambda_n\nearrow Q(0)$ as $n\to\infty$.
    \end{enumerate}
  \end{example}

\begin{remark}
  As we shall see, the property \ref{item:4} is the one which
  causes concentration. Concerning \ref{item:3}, we do not
  exclude the case of $K_n=0$, i.e., $V_n=V$ for all $n$.
\end{remark}

Let $E :=H^1_0(\Omega)$. According to \ref{item:3},
\begin{equation*}
  \norm{u} := \left(\io (\abs{\nabla u}^2+Vu^2)\,dx\right)^{1/2}
\end{equation*}
is an equivalent norm in $E$. The notation $\norm{\cdot}$ will
always refer to this norm. We also set
\begin{gather*}
  \norm{u}_n := \left(\io (\abs{\nabla u}^2+V_nu^2)\,dx\right)^{1/2}, \\
  \abs{u}_{q,A}:=\xlr(){\int_A\abs{u}^q\,dx}^{1/q},
\end{gather*}
$\abs{u}_{\infty,A} := \esssup_A\abs{u}$, and we abbreviate
$\abs{u}_{q,\Omega}$ to $\abs{u}_q$.  For $r>0$ and $a\in\rn$, we
put
\begin{equation*}
  B_r(a)\coloneqq\{x\in\dR^N: \abs{x-a}<r\}.
\end{equation*}
Weak convergence will be denoted by ``$\,\rh\,$''.

\medskip

In Section \ref{sec:conc-h1-norm} we show that \eqref{eq:21} has
a ground state solution and that any sequence of solutions
$(u_n)$ to \eqref{eq:21} concentrates at the origin in the $H^1$-
and the $L^p$-norm. In Section \ref{sec:conc-lq-norm}
concentration in the $L^q$-norms for different $q$ is considered
and in Section \ref{sec:conc-sever-regi} it is shown that if
$Q_n$ is positive in a neighbourhood of a finite number of
points, then ground states concentrate at one of these points.

\section{Concentration in the $H^1$- and $L^p$-norms}
\label{sec:conc-h1-norm}

\begin{proposition} \label{norm} For all $n$ large enough,
  $\norm{\cdot}_n$ is a uniformly equivalent norm in $E$, i.e.,
  there exist constants $c_1,c_2>0$ and $N_0\ge 1$ such that
  \begin{equation*}
    c_1\norm{u} \le \norm{u}_n \le c_2\norm{u} \quad \text{for all }u\in E
    \text{ and } n\ge N_0.
  \end{equation*}
\end{proposition}

In what follows we always assume $n$ is so large that the
conclusion of this proposition holds.

\begin{proof}
  Let $\cK_n: E\to E$ be the linear operator given by
  \begin{equation*}
    \scp{\cK_nu, v} := \io K_nuv\,dx.
  \end{equation*}
  Using the H\"older and Sobolev inequalities we see that for
  each $\eps>0$ there is $N_\eps$ such that
  \begin{align*}
    \xabs{\scp{ \cK_nu,v}} \le
    C_1\int_{B_\eps(0)}\abs{uv}\,dx
    & \le C_1\abs{B_\eps(0)}^{(q-2)/q}\abs{u}_q\abs{v}_q \\
    & \le C_2\abs{B_\eps(0)}^{(q-2)/q}\norm{u}\,\norm{v} \quad \text{for
      all } n\ge N_\eps,
  \end{align*}
  where $q=2^*$ if $N\ge 3$, $q>2$ if $N=1$ or 2, $\abs{B_\eps(0)}$
  denotes the measure of $B_\eps(0)$ and $C_1,C_2$ are constants
  independent of $\eps$ and $n$. Now the conclusion easily
  follows by taking $\eps$ small enough.
\end{proof}

Next we prove our main existence result for \eqref{eq:21}.

\begin{theorem} \label{thm21} Suppose that $V_n$ and $Q_n$
  satisfy \ref{item:3}, \ref{item:4} above and
  $p\in(2,2^*)$. Then for all sufficiently large $n$ problem
  \eqref{eq:21} has a positive ground state solution $u_n\in
  E$. Moreover, there exists a constant $\alpha>0$, independent
  of $n$, such that $\norm{u_n}\ge\alpha$.
\end{theorem}

\begin{proof}
  Let $J_n(v):= \io Q_n\abs{v}^p\,dx$ and
  \begin{equation*}
    s_n := \inf_{J_n(v)>0}\frac{\norm{v}_n^2}{J_n(v)^{2/p}} \equiv
    \inf_{J_n(v)>0}\frac{\io(\abs{\nabla v}^2+V_nv^2)\,dx}{\left(\io
        Q_n\abs{v}^p\,dx\right)^{2/p}}.
  \end{equation*}
  If the infimum is attained at $v_n$, then it follows via the
  Lagrange multiplier rule that $u_n=c_nv_n$ is a solution of
  \eqref{eq:21} for an appropriate $c_n>0$. Moreover, since $v_n$
  may be replaced by $\abs{v_n}$, we may assume $v_n\ge 0$ (and hence
  $u_n\ge 0$). To show that $u_n>0$, we note that $u_n$ satisfies
  \begin{equation*}
    -\Delta v + (V(x)+Q_n^-(x)u_n(x)^{p-2})v =
    Q_n^+(x)u_n(x)^{p-1}\ge 0,
  \end{equation*}
  where $Q_n^\pm(x) := \max\{\pm Q_n(x),0\}$. Since
  $V(x)+Q_n^-(x)u_n(x)^{p-2}\ge 0$, it follows from the strong
  maximum principle (see e.g.\ \cite[Theorem~8.19]{MR737190})
  that $v_n>0$ (in fact it can be shown that all ground states
  have constant sign).

  If $u_n\ne 0$ is a solution to \eqref{eq:21}, then, multiplying
  the equation by $u_n$, integrating by parts and using the
  Sobolev inequality, we obtain
  \begin{equation} \label{bound} \norm{u_n}_n^2 = \io Q_n\abs{u_n}^p\,dx
    \le C_1\abs{u_n}_p^p \le C_2\norm{u_n}_n^p,
  \end{equation}
  hence according to Proposition \ref{norm}, $\norm{u_n}\ge\alpha$
  for some $\alpha>0$ and all large~$n$.

  It remains to show that the infimum is attained. Let $(v_k)$ be
  a minimizing sequence for $s_n$, normalized by
  $J_n(v_k)=1$. Then $\norm{v_k}_n$ is bounded, so we may assume
  passing to a subsequence that $v_k\rh v$ in $E$ and $v_k(x)\to
  v(x)$ a.e.\ in $\Omega$. Since the norm is lower semicontinuous
  and $Q_n<0$ on $\abs{x}>1$ for $n$ large, it follows from the
  Rellich-Kondrachov theorem and Fatou's lemma (applied on the
  set $\abs{x}>1$) that
  \begin{align*}
    s_n = \lim_{k\to\infty}\norm{v_k}_n^2 & =
    \lim_{k\to\infty}\frac{\norm{v_k}_n^2}{\left(\int_{\abs{x}<1}Q_n\abs{v_k}^p\,dx
        +\int_{\abs{x}>1}Q_n\abs{v_k}^p\,dx\right)^{2/p}} \\
    & \ge \frac{\norm{v}_n^2}{J_n(v)^{2/p}} \ge s_n.
  \end{align*}
  Thus $v$ is a minimizer.
\end{proof}

Note that the only properties of $V_n$ and $Q_n$ which are
essential in this proof are that $\norm{\cdot}_n$ is a norm, $Q_n\in
L^\infty(\Omega)$, $Q_n>0$ on a set of positive measure and
$Q_n(x)\le 0$ for all $\abs{x}$ large enough.

\begin{remark} \label{rem25}
  \begin{enumerate}[label=\textup{(\alph{*})}]
  \item We see from \eqref{bound} that $\norm{u}\ge\alpha$ for any
    nontrivial solution $u$ of \eqref{eq:21} provided $n$ is
    large enough.
  \item Since the Krasnoselskii genus of the manifold $J_n(v)=1$
    is infinite and the functional $v\mapsto \io Q_n^+\abs{v}^p\, dx$
    is weakly continuous, it is not difficult to see using
    standard minimax methods that \eqref{eq:21} has infinitely
    many solutions. Since we shall not use this result, we leave
    out the details.
  \item The observation that $Q<0$ outside a large ball implies
    compactness (and thus existence of solutions) seems to go
    back to \cite{MR1861096}.
  \end{enumerate}
\end{remark}

In the sequel suppose for each $n$ that $u_n$ is a nontrivial
solution of \eqref{eq:21} and set $w_n:=u_n/\norm{u_n}_n$.

\begin{lemma} \label{lem22} $\norm{u_n}\to\infty$ as $n\to\infty$.
\end{lemma}

\begin{proof}
  Assuming the contrary, $u_n\rh u$ in $E$ and $u_n\to u$ in
  $L^p_{loc}(\Omega)$ after passing to a subsequence. Multiplying
  \eqref{eq:21} (with $u=u_n$) by $u_n$, integrating and using
  the fact that $Q_n<0$ for each $\eps>0$ and $n\ge N_\eps$, we
  obtain
  \begin{align*}
    \limsup_{n\to\infty} \norm{u_n}_n^2
    & = \limsup_{n\to\infty}\io Q_n\abs{u_n}^p\,dx\\
    & \le \limsup_{n\to\infty}\int_{\abs{x}<\eps} Q_n\abs{u_n}^p\,dx \le
    C \int_{\abs{x}<\eps}\abs{u}^p\,dx.
  \end{align*}
  Letting $\eps\to 0$ and using Proposition \ref{norm}, we see
  that $u_n\to 0$ in $E$, a contradiction because
  $\norm{u_n}\ge\alpha>0$.
\end{proof}

\begin{lemma} \label{lem23} $w_n\weakto0$ in $E$ as $n\to\infty$.
\end{lemma}

\begin{proof}
  Passing to a subsequence we may assume that $w_n\rh w$ in $E$.
  Multiplying \eqref{eq:21} (with $u=u_n$) by $u_n/\norm{u_n}_n^2$,
  we obtain
  \begin{equation}\label{eq:5}
    1 = \norm{w_n}_n^2 = \norm{u_n}_n^{p-2}\io Q_n\abs{w_n}^p\,dx.
  \end{equation}
  By Lemma \ref{lem22}, $\io Q_n\abs{w_n}^p\,dx \to 0$. Let
  $0<\eps<\eps_1$. Then
  \begin{align*}
    0 = \lim_{n\to\infty} \io Q_n\abs{w_n}^p\,dx
    & =  \lim_{n\to\infty}\left(\int_{\abs{x}<\eps}Q_n\abs{w_n}^p\,dx + \int_{\abs{x}>\eps}Q_n\abs{w_n}^p\,dx\right)  \\
    & \le  \lim_{n\to\infty}\left(\int_{\abs{x}<\eps}Q_n\abs{w_n}^p\,dx + \int_{\abs{x}>\eps_1}Q_n\abs{w_n}^p\,dx\right) \\
    & \le C \int_{\abs{x}<\eps}\abs{w}^p\,dx - \delta_{\eps_1}
    \int_{\abs{x}>\eps_1}\abs{w}^p\,dx.
  \end{align*}
  If $w\ne 0$, we may chose $\eps_1$ so small that the second
  integral on the right-hand side above is positive. Letting
  $\eps\to 0$, we obtain a contradiction.
\end{proof}

Now we can study concentration of $(u_n)$ as $n\to\infty$. Let
$\eps>0$ be given and let $\chi\in C^\infty(\Omega,[0,1])$ be
such that $\chi(x)=0$ for $x\in B_{\eps/2}(0)$ and $\chi(x)=1$
for $x\notin B_\eps(0)$. Multiplying \eqref{eq:21} (with $u=u_n$)
by $\chi u_n$ we obtain
\begin{equation*}
  \io(\nabla u_n\cdot\nabla(\chi u_n)+\chi V_nu_n^2)\,dx = \io \chi
  Q_n\abs{u_n}^p\,dx,
\end{equation*}
or equivalently,
\begin{equation*}
  \io \chi(\abs{\nabla u_n}^2+V_nu_n^2)\,dx - \io \chi Q_n\abs{u_n}^p\,dx =
  -\io u_n\nabla\chi\cdot \nabla u_n\,dx.
\end{equation*}
Given $\eps>0$, we have $Q_n\le-\delta_\eps$ and $V_n=V\ge 0$ on
$\supp\chi$ provided $n$ is large enough. Hence for all such $n$,
\begin{multline} \label{eq21}
  \int_{\Omega\setminus B_\eps(0)}(\abs{\nabla u_n}^2+V_nu_n^2)\,dx +\delta_\eps\int_{\Omega\setminus B_\eps(0)}\abs{u_n}^p\,dx \\
  \begin{aligned}
    &\quad \le \io \chi(\abs{\nabla u_n}^2+V_nu_n^2)\,dx - \io \chi Q_n\abs{u_n}^p\,dx  \\
    & \quad \le d_\eps \int_{B_\eps(0)\setminus
      B_{\eps/2}(0)}\abs{u_n}\,\abs{\nabla u_n}\,dx, 
  \end{aligned}
\end{multline}
where $d_\eps$ is a constant independent of $n$. Since $w_n =
u_n/\norm{u_n}_n \to 0$ in $L^2_{loc}(\Omega)$ according to Lemma
\ref{lem23}, it follows from H\"older's inequality that
\begin{equation*}
  \int_{B_\eps(0)\setminus B_{\eps/2}(0)}\abs{w_n}\,\abs{\nabla w_n}\,dx
  \to 0.
\end{equation*}
So (\ref{eq21}) implies
\begin{equation} \label{eq22} \lim_{n\to\infty} \left(
    \int_{\Omega\setminus B_\eps(0)}(\abs{\nabla w_n}^2+V_nw_n^2)\,dx
    +\delta_\eps \norm{u_n}_n^{p-2}\int_{\Omega\setminus
      B_\eps(0)}\abs{w_n}^p\,dx \right) = 0.
\end{equation}

\begin{theorem} \label{thm24} Suppose that $V_n$ and $Q_n$
  satisfy \textup{\ref{item:3}}, \textup{\ref{item:4}} and
  $p\in(2,2^*)$. Let $u_n$ be a nontrivial solution for
  \eqref{eq:21} and let $w_n = u_n/\norm{u_n}_n$. Then for every
  $\varepsilon>0$ it holds that
  \begin{align} \label{eq25} \lim_{n\to\infty}
    \int_{\Omega\setminus B_\eps(0)}(\abs{\nabla w_n}^2+V_nw_n^2)\,dx
    &= 0 \\\intertext{and}
    \lim_{n\to\infty}\norm{u_n}_n^{p-2}\int_{\Omega\setminus
      B_\eps(0)}\abs{w_n}^p\,dx &= 0.\label{eq:6}
  \end{align}
  Moreover,
  \begin{equation*}
    \lim_{n\to\infty} \frac{\int_{\Omega\setminus
        B_\eps(0)}(\abs{\nabla u_n}^2+V_nu_n^2)\,dx}{\io (\abs{\nabla
      u_n}^2+V_nu_n^2)\,dx} = 0 \quad\text{and}\quad
    \lim_{n\to\infty} \frac{\int_{\Omega\setminus
        B_\eps(0)}\abs{u_n}^p\,dx}{\io \abs{u_n}^p\,dx} = 0.
  \end{equation*}
\end{theorem}

\begin{proof}
  The first conclusion is an immediate consequence of
  \eqref{eq22}. Since
  \begin{equation*}
    \io (\abs{\nabla w_n}^2+V_nw_n^2)\,dx \equiv \norm{w_n}_n^2 = 1,
  \end{equation*}
  it follows from \eqref{eq25} that
  \begin{equation*}
    \lim_{n\to\infty} \frac{\int_{\Omega\setminus
        B_\eps(0)}(\abs{\nabla u_n}^2+V_nu_n^2)\,dx}{\io (\abs{\nabla
      u_n}^2+V_nu_n^2)\,dx} = \lim_{n\to\infty}
    \frac{\int_{\Omega\setminus B_\eps(0)}(\abs{\nabla
      w_n}^2+V_nw_n^2)\,dx}{\io (\abs{\nabla w_n}^2+V_nw_n^2)\,dx} = 0.
  \end{equation*}
  By \eqref{eq:5}
  \begin{equation*}
    C\norm{u_n}_n^{p-2}\io\abs{w_n}^p\,dx \ge \norm{u_n}_n^{p-2}\io
    Q_n\abs{w_n}^p\,dx = \norm{w_n}_n^2 = 1.
  \end{equation*}
  This and \eqref{eq:6} imply
  \begin{equation*}
    \lim_{n\to\infty} \frac{\int_{\Omega\setminus
        B_\eps(0)}\abs{u_n}^p\,dx}{\io \abs{u_n}^p\,dx} = \lim_{n\to\infty}
    \frac{\norm{u_n}_n^{p-2}\int_{\Omega\setminus
        B_\eps(0)}\abs{w_n}^p\,dx}{\norm{u_n}_n^{p-2}\io \abs{w_n}^p\,dx} = 0.
  \end{equation*}
\end{proof}

\section{Concentration in the $L^q$-norm}
\label{sec:conc-lq-norm}

Here we consider concentration in other norms.  Note that
\begin{equation*}
  \frac{N(p-2)}{2}<p
  \qquad\text{and}\qquad \text{if } N\ge 3, \text{ then } \frac{2N-2}{N-2}<2^*.
\end{equation*}

\begin{theorem}
  \label{prop:conc-lq-norm} Suppose that \ref{item:3} and
  \ref{item:4} hold and there exist $R,\lambda>0$ such that
  $V\ge\lambda$ whenever $x\in\Omega\setminus B_R(0)$.  For every
  $n\in\dN$ let $u_n$ denote a non\-trivial solution to
  \eqref{eq:21}.  If $\varepsilon>0$ is such that
  $\olB_\varepsilon(0)\subset\Omega$, then the following hold:
  \begin{enumerate}[label=\textup{(\alph{*})}]
  \item \label{item:5} For every $q\in[1,\infty]$ the norm
    $\abs{u_n}_{q,\Omega\ssm B_\varepsilon(0)}$ remains bounded,
    uniformly in $n$.
  \item \label{item:2} If $\delta=\delta_{\varepsilon}>0$ in
    \ref{item:4} can be chosen independently of $\varepsilon>0$,
    if $N\ge3$ and $p\in\xlr[){\frac{2N-2}{N-2},2^*}$, then
    $\lim_{n\to\infty} \abs{u_n}_{q,\Omega \ssm
      B_\varepsilon(0)}=0$, for every $q\in[1,\infty]$.
  \item \label{item:1} For every $q\ge 1$, $q\in \biglr(]
    {\frac{N(p-2)}{2}, \infty}$ it holds that $\lim_{n\to\infty}
    \abs{u_n}_q= \infty$ and
    \begin{equation}
      \label{eq:1}
      \lim_{n\to\infty} 
      \frac{\abs{u_n}_{q,\Omega\ssm B_\varepsilon(0)}}{\abs{u_n}_{q}}=0.
    \end{equation}
  \item \label{item:8} If $\frac{N(p-2)}2 \ge 1$, then for $q=\frac{N(p-2)}{2}$ it holds that
    \begin{equation}
      \label{eq:11}
      \liminf_{n\to\infty}\abs{u_n}_q>0.
    \end{equation}
    If the hypotheses in \ref{item:2} are satisfied, then
    \eqref{eq:1} holds for this $q$.
  \end{enumerate}
\end{theorem}

Note that $V\ge\lambda>0$ for $x\in \Omega\setminus B_R(0)$ is
trivially satisfied if $\Omega$ is bounded and $R$ large
enough. Note also that it follows from the Poincar\'e inequality
that the above condition and $V\ge 0$ for all $x$ imply
$\sigma(-\Delta+V)\subset(0,\infty)$.

\begin{remark}   \label{rem:kerr}
  It holds that
  \begin{equation}
    \label{eq:4}
    \frac{N(p-2)}{2}<2\qquad\text{iff}\qquad
    p<2+\frac{4}{N}.
  \end{equation}
  In that case $(u_n)$ concentrates with respect to the
  $L^q$-norm for every $q\in[2,\infty]$, covering the physically
  interesting $L^2$-concentration. 
  In particular, for the travelling planar waves considered in
  the introduction we have $N=1$ or 2.  In a Kerr medium, where $p=4$, \eqref{eq:4} and
  Theorem~\ref{prop:conc-lq-norm} yield concentration
  near $0$ with respect to the $L^2$-norm for $N=1$ but not for $N=2$.
\end{remark}
\begin{proof}[Proof of Theorem~\ref{prop:conc-lq-norm}]
  To prove \ref{item:5}, fix $\delta_{\varepsilon/2}>0$ and
  $N_{\varepsilon/2}$ as in \ref{item:4}.  By
  \cite[Sect.~1.6]{MR1226934} there is a positive classical
  solution $w$ of the equation
  \begin{equation*}
    -\Delta u=-\delta_{\varepsilon/2}\abs{u}^{p-2}u
  \end{equation*}
  on $\dR^N\ssm \olB_{\varepsilon/2}(0)$ that satisfies
  $\lim_{\abs{x}\to\varepsilon/2}w(x)=\infty$ and
  $\lim_{\abs{x}\to\infty}w(x)=0$.  Fixing $n\ge
  N_{\varepsilon/2}$, setting $z_n\coloneqq w-u_n$ and
  \begin{equation*}
    \varphi_n(x)
    \coloneqq(p-1)\int_0^1\abs{sw(x)+(1-s)u_n(x)}^{p-2}\dint s
    \ge0
  \end{equation*}
  we obtain 
 \begin{align*}
  \vp_nz_n &  =  (p-1) \int_0^1 \abs{sw+(1-s)u_n}^{p-2}(w-u_n)\,ds \\
  & =  \int_0^1\frac d{ds}\left(\abs{sw+(1-s)u_n}^{p-2}(sw+(1-s)u_n)\right)ds 
   = w^{p-1}-\abs{u_n}^{p-2}u_n
\end{align*}
   and hence
  from \ref{item:4} 
  \begin{equation*}
    (-\Delta+V-Q_n \varphi_n)z_n
    =-\Delta w+Vw-Q_n w^{p-1} 
    \ge-\Delta w+\delta_{\varepsilon/2} w^{p-1}=0.
  \end{equation*}
  Note that $V-Q_n\varphi_n\ge0$ in
  $\Omega\ssm\olB_{\varepsilon/2}(0)$ since $n\ge
  N_{\varepsilon/2}$.  By the continuity of $u_n$ and since
  $w_n(x)\to\infty$ as $x\to\partial B_{\varepsilon/2}(0)$, there
  is $r\in(\varepsilon/2,\varepsilon)$ such that $z_n\ge 0$ on
  $\partial B_r(0)$.  Moreover, $z_n\ge0$ on $\partial\Omega$.
  If $\Omega$ is bounded then we may apply the maximum principle
  for weak supersolutions \cite[Theorem~8.1]{MR737190} to $z_n$
  and obtain $z_n\ge0$ in $\Omega\ssm B_r(0)$.  If $\Omega$ is
  unbounded, we consider any $\gamma>0$ and pick $\wt R>0$ such
  that $z_n\ge-\gamma$ in $\Omega\ssm B_{\wt R}(0)$.  This is
  possible since $w(x)$ tends to 0 as $\abs{x}\to\infty$ by
  construction.  Moreover, $u_n\in E$ and standard estimates from
  regularity theory imply that also $u_n(x)\to0$ as
  $\abs{x}\to\infty$.  Now the same maximum principle, applied on
  $\Omega\cap(B_{\wt R}(0)\ssm\olB_r(0))$, implies
  $z_n\ge-\gamma$ in all of $\Omega\ssm B_r(0)$.  Letting
  $\gamma\to0$ we obtain $z_n\ge0$ also in this case.  In an
  analogous way we obtain $u_n\ge-w$ (take $z_n:=w+u_n$), and
  hence
  \begin{equation*}
    \abs{u_n}\le w\qquad\text{in }\Omega\ssm
    B_\varepsilon(0),\ \text{for all }n\ge N_{\varepsilon/2}.
  \end{equation*}
  Note that $w$ is continuous in $\overline{\Omega}\ssm
  B_\varepsilon(0)$. Hence \ref{item:5} follows if $\Omega$ is
  bounded. For unbounded $\Omega$, according to Lemma \ref{exp}
  below, setting $M := \max_{\abs{x}=R}w$ we obtain $\abs{u_n} \le
  Me^{-\alpha\abs{x-R}}$ whenever $x\in \Omega\setminus B_R(0)$. So
  the conclusion in \ref{item:5} holds also in this case.

  Next we consider \ref{item:2}.  The hypotheses imply that there
  is $\delta>0$ such that $Q_n\le-\delta$ on $\Omega\ssm
  B_{1/n}(0)$ for every $n$ large enough.  Denote by $w_n$ a
  positive solution of
  \begin{equation}
    \label{eq:8}
    -\Delta u=-\delta\abs{u}^{p-2}u
  \end{equation}
  on $\dR^N\ssm B_{1/n}(0)$ with boundary conditions
  $\lim_{\abs{x}\to1/n} w_n(x) =\infty$ and
  $\lim_{\abs{x}\to\infty} w_n(x)=0$, as before.  Then the
  sequence $w_n$ is monotone decreasing since $w_n\ge w_{n+1}$ on
  $B_{1/n}(0)$ for every $n\in\dN$ by the maximum principle
  (using similar arguments as before).  Therefore $w_n$
  converges locally uniformly to a nonnegative solution $w$ of
  \eqref{eq:8} on $\dR^N\ssm\{0\}$.  Our hypotheses on $N$ and
  $p$, and \cite[Theorem~2]{MR592099} imply that $w$ extends to
  an entire solution of \eqref{eq:8}.  By
  \cite[Theorem~1.3]{MR1226934} $w\equiv0$.  On the other hand,
  the function $w_{n}$ dominates the solution $u_n$ on
  $\overline{\Omega}\ssm B_r(0)$ for some
  $r\in(\varepsilon/2,\varepsilon)$ and large $n$, as seen in the
  proof of \ref{item:5}.  Therefore also $u_n$ converges to $0$
  locally uniformly in $\Omega\ssm B_r(0)$.  Together with
  Lemma~\ref{exp} (take $M:=\max_{\abs{x}=R}w_n$) we obtain
  $\lim_{n\to\infty} \abs{u_n}_{q,\Omega\ssm \olB_\varepsilon(0)}
  =0$.

  In the proof of \ref{item:1} first consider the case $q\ge 1$,
  $q\in(N(p-2)/2,p]$.  Since $u_n$ is a solution, by
  \ref{item:3}, H\"older's inequality, the Sobolev embedding, and
  Proposition~\ref{norm} we have
  \begin{equation}\label{eq:10}
    \norm{u_n}_n^2
    =\int_\Omega Q_n\abs{u_n}^p
    \le C_1\abs{u_n}_p^p
    \le C_1\abs{u_n}_q^{p\theta}\abs{u_n}_{2^*}^{p(1-\theta)}
    \le C_2\abs{u_n}_q^{p\theta}\norm{u_n}_n^{p(1-\theta)}.
  \end{equation}
  Here $C_1,C_2$ are independent of $n$, and $\theta$ satisfies
  \begin{equation*}
    \frac1p=\frac\theta q+\frac{1-\theta}{2^*}.
  \end{equation*}
  From Lemma~\ref{lem22} we see that it is sufficient to impose
  $p(1-\theta)<2$ or, equivalently, $q>N(p-2)/2$.  This and
  \ref{item:5} prove the case $q \in(N(p-2)/2,p]$.

  Since we already know from \eqref{eq:10} that
  $\abs{u_n}_p\to\infty$, \ref{item:5} yields
  $\abs{u_n}_{p,B_\varepsilon(0)} \to\infty$ and hence
  $\abs{u_n}_{q,B_\varepsilon(0)} \to\infty$ as $n\to\infty$, for
  every $q\in[p,\infty]$.  Now \eqref{eq:1} follows from
  \ref{item:5}.

  To prove \ref{item:8} we note that \eqref{eq:10} implies
  \eqref{eq:11} for $q=\frac{N(p-2)}{2}$.  The other claim is
  obvious.
\end{proof}

\begin{lemma} \label{exp} Suppose $\Omega$ is unbounded and
  $V(x)\ge \lambda > 0$ for $x\in\Omega\setminus B_R(0)$. If
  $u_n$ is a nontrivial solution to \eqref{eq:21} and $\abs{u_n}\le
M$  on $\partial B_R(0)$, then $\abs{u_n} \le Me^{-\alpha\abs{x-R}}$ for $x\in\Omega\setminus
  B_R(0)$, where $\alpha:=\sqrt\lambda$.
\end{lemma}

\begin{proof} 
  We follow the argument in \cite[Proposition
  4.4]{MR960098}. Write $u=u_n$ and let
  \begin{gather*}
    W(x) := Me^{-\alpha (\abs{x}-R)}, \\
    \Omega_S := \{ x\in\Omega: R<\abs{x} <S \text{ and } u(x) >W(x)
    \}.
  \end{gather*}
  Condition \ref{item:4} implies that there is $\delta>0$ such
  that for $x\in\Omega_S$ we have $u(x)>0$ and
  \begin{equation*}
  -\Delta u \le -V(x)u-\delta\abs{u}^{p-2}u \le -\lambda u,
  \end{equation*}
  hence
  \begin{equation} \label{ineq1} \Delta (W-u)=\left( \alpha^2-
      \frac { \alpha (N-1)}{\abs{x}} \right)W-\Delta u \leq
    \alpha^2(W-u) \leq 0
  \end{equation}
  for such $x$. By the maximum principle,
  \begin{equation*}
  W(x)-u(x) \geq \min_{x \in \partial \Omega_S} (W-u) \geq \min
  \left\{ 0, \min_{\abs{x} =S}(W-u)\right\}.
  \end{equation*}
  Since $\lim_{\abs{x} \to \infty} u(x)= \lim_{\abs{x} \to \infty}
  W(x)=0$, letting $S \to \infty$ we obtain
  \begin{equation*}
  u(x) \leq W(x) = Me^{-\alpha (\abs{x}-R)} \quad \text{for} \quad
  x\in \Omega\setminus B_R(0).
  \end{equation*}
  Replacing $u(x)>W(x)$ by $u(x)<-W(x)$ in the definition of
  $\Omega_S$ and $W-u$ by $W+u$ in \eqref{ineq1}, we see that
  $u\ge -W$ for $x\in\Omega\setminus B_R(0)$.
\end{proof}

\begin{remark}
  In the proof of \ref{item:2} it was essential that \eqref{eq:8}
  has no nontrivial solution $w\ge 0$ in $\rn\setminus\{0\}$. If
  $2<p<(2N-2)/(N-2)$, this argument cannot be used because
  $w=c_p\abs{x}^{-2/(p-2)}$ is a solution of \eqref{eq:8} for a
  suitable constant $c_p>0$ (note that if $p>(2N-2)/(N-2)$, then
  $w$ solves equation \eqref{eq:8} with $\delta$ replaced by
  $-\delta$).
\end{remark}

\section{Concentration at several points}
\label{sec:conc-sever-regi}

In this section we assume that the functions $Q_n$ are positive
in a neighbourhood of two distinct points $x_1,x_2\in\Omega$ and
$V_n$ may not be equal to $V$ in this neighbourhood. More
precisely, we assume

\begin{enumerate}[label=(A\arabic{*})]\setcounter{enumi}{2}
\item \label{item:6} $V\in L^\infty(\Omega)$, $V\ge 0$ and
  $\sigma(-\Delta+V)\subset(0,\infty)$, where the spectrum
  $\sigma$ is realized in $H^1_0(\Omega)$. $V_n = V+K_n$, where
  $K_n\in L^\infty(\Omega)$, and there exists a constant $B$ such
  that $\norm{K_n}_\infty \le B$ for all $n$. Moreover, for each
  $\eps>0$ there is $N_\eps$ such that $\supp K_n\subset
  B_\eps(x_1)\cup B_\eps(x_2)$ whenever $n\ge N_\eps$.
\item \label{item:7} $Q_n\in L^\infty(\Omega)$, $Q_n>0$ in a
  neighbourhood of $\{x_1\}\cup\{x_2\}$ and there exists a
  constant $C$ such that $\norm{Q_n}_\infty\le C$ for all
  $n$. Moreover, for each $\eps>0$ there exist constants
  $\delta_\eps>0$ and $N_\eps$ such that $Q_n\le -\delta_\eps$
  whenever $x\notin B_\eps(x_1)\cup B_\eps(x_2)$ and $n\ge
  N_\eps$.
\end{enumerate}

We have taken two points $x_1,x_2$ for notational convenience
only. The arguments below are valid for any finite number of
points in $\Omega$.

It is clear that the arguments of Section \ref{sec:conc-h1-norm}
go through with obvious changes if one replaces
\ref{item:3}-\ref{item:4} by \ref{item:6}-\ref{item:7}. Our
purpose here is to show that if \ref{item:6}-\ref{item:7} hold,
then each ground state $u_n$ for $n$ large concentrates exactly
at one of the points $x_1, x_2$.  In Section
\ref{sec:conc-h1-norm} $u_n$ could be any nontrivial solution to
\eqref{eq:21}. To the contrary, in Theorem \ref{thm41} below it
is important that $u_n$ is a ground state.

As in Section \ref{sec:conc-h1-norm}, we put $J_n(u) = \io
Q_n\abs{u}^p\,dx$ and
\begin{equation*}
s_n := \inf_{J_n(u)>0}\frac{\norm{u}_n^2}{J_n(u)^{2/p}} \equiv
\inf_{J_n(u)>0}\frac{\io(\abs{\nabla u}^2+V_nu^2)\,dx}{\left(\io
    Q_n\abs{u}^p\,dx\right)^{2/p}}.
\end{equation*}

\begin{theorem} \label{thm41} Suppose that $V_n$ and $Q_n$
  satisfy \ref{item:6}, \ref{item:7} and $p\in(2,2^*)$. Let $u_n$
  be a ground state solution for \eqref{eq:21}. Then, for $n$
  large, $u_n$ concentrates at $x_1$ or $x_2$. More precisely,
  for each $\eps>0$ we have, passing to a subsequence,
  \begin{equation} \label{eq41}
    \lim_{n\to\infty}\frac{\int_{\Omega\setminus
        B_\eps(x_j)}(\abs{\nabla u_n}^2+V_nu_n^2)\,dx}{\io(\abs{\nabla
      u_n}^2+V_nu_n^2)\,dx} = 0 \text{ and
    }\lim_{n\to\infty}\frac{\int_{\Omega\setminus
        B_\eps(x_j)}Q_n\abs{u_n}^p\,dx}{\io Q_n\abs{u_n}^p\,dx} = 0
  \end{equation} 
  for $j=1$ or 2 (but not for $j=1$ and 2).
\end{theorem}

\begin{remark} \label{rem42} Note that in view of the obvious
  modification of Theorem~\ref{thm24} the limits in \eqref{eq41}
  are 0 if $\Omega\setminus B_\eps(x_j)$ is replaced by
  $\Omega\setminus(B_\eps(x_1)\cup B_\eps(x_2))$. So if $j=1$ in
  \eqref{eq41}, then concentration occurs at $x_1$ and if $j=2$,
  it occurs at $x_2$.
\end{remark}

\begin{proof}[Proof of Theorem \ref{thm41}]
  Renormalizing, we may assume that $J_n(u_n)=\io
  Q_n\abs{u_n}^p\,dx=1$ (then $u_n$ may not be a solution of
  \eqref{eq:21} but we still have $s_n =
  \norm{u_n}^2/J_n(u_n)^{2/p}$). Let $\xi_j\in
  C_0^\infty(\Omega,[0,1])$ be a function such that $\xi_j=1$ on
  $B_{\eps/2}(x_j)$ and $\xi_j=0$ on $\Omega\setminus
  B_{\eps}(x_j)$, $j=1,2$, where $\eps$ is so small that
  $\olB_{\eps}(x_j) \subset \Omega$ and $\olB_{\eps}(x_1)\cap
  \olB_{\eps}(x_2) = \emptyset$. Set $v_n:=\xi_1u_n$,
  $w_n:=\xi_2u_n$, $z_n:=u_n-v_n-w_n$. Since $\supp z_n\subset
  \Omega\setminus(B_{\eps/2}(x_1)\cup B_{\eps/2}(x_2))$ and the
  conclusion of Theorem \ref{thm24} remains valid after an
  obvious modification, we have
  \begin{align*}
    \norm{u_n}_n^2 & = \io(\abs{\nabla u_n}^2 +V_nu_n^2)\,dx \\
    & = \left(\io(\abs{\nabla v_n}^2+V_nv_n^2)\,dx + \io(\abs{\nabla w_n}^2+V_nw_n^2)\,dx \right)(1+o(1)) \\
    & = (\norm{v_n}_n^2+\norm{w_n}_n^2)(1+o(1))
  \end{align*}
  and
  \begin{align*}
    1 = J_n(u_n) \!=\! \io Q_n\abs{u_n}^p\,dx & = \io Q_n\abs{v_n}^p\,dx + \io Q_n\abs{w_n}^p\,dx +o(1) \\
    & = J_n(v_n)+J_n(w_n)+o(1).
  \end{align*}
  Assume first that $\limsup_{n\to\infty}J_n(v_n)\ge 0$ and
  $\limsup_{n\to\infty}J_n(w_n)\ge 0$. Then, passing to a
  subsequence, $J_n(v_n)\to c_0\in[0,1]$ and $J_n(w_n)\to
  1-c_0\in [0,1]$. Suppose $c_0\in(0,1)$. Since $p>2$, for $n$
  large enough we have
  \begin{align*}
    s_n & = \frac{\norm{u_n}_n^2}{J_n(u_n)^{2/p}}  = \frac{(\norm{v_n}_n^2+\norm{w_n}_n^2)(1+o(1))}{(J_n(v_n)+J_n(w_n)+o(1))^{2/p}} \\
    & >
    \frac{\norm{v_n}_n^2+\norm{w_n}_n^2}{J_n(v_n)^{2/p}+J_n(w_n)^{2/p}}\ge
    \min\left\{\frac{\norm{v_n}_n^2}{J_n(v_n)^{2/p}},
      \frac{\norm{w_n}_n^2}{J_n(w_n)^{2/p}}\right\} \ge s_n,
  \end{align*}
  a contradiction. So $c_0=0$ or 1. If $c_0=1$ (say), then the
  second limit in (\ref{eq41}) is 0 for $j=1$ because $\supp v_n
  \subset B_\eps(x_1)$. Also the first limit is 0 since otherwise
  $\norm{w_n}_n^2/\norm{v_n}_n^2$ is bounded away from 0 for large $n$,
  and we obtain
  \begin{equation} \label{eq42} s_n =
    \frac{(\norm{v_n}_n^2+\norm{w_n}_n^2)(1+o(1))}{(J_n(v_n)+J_n(w_n)+o(1))^{2/p}}
    > \frac{\norm{v_n}_n^2}{J_n(v_n)^{2/p}} \ge s_n,
  \end{equation} 
  a contradiction again.

  Finally, suppose $\limsup_{n\to\infty}J_n(w_n) < 0$ (the case
  $\limsup_{n\to\infty}J_n(v_n) < 0$ is of course
  analogous). Passing to a subsequence, $J_n(w_n)\le -\eta$ for
  some $\eta>0$ and all $n$ large enough. Then \eqref{eq42} holds
  for such $n$ because $ J_n(v_n) > J_n(v_n)+J_n(w_n)+o(1)$.
\end{proof}

\bibliographystyle{amsplain-abbrv.bst} \bibliography{areiszu-1}
\subsubsection*{Contact information:}
\begin{sloppypar}
  \begin{description}
  \item[Nils Ackermann:] Instituto de Matem\'{a}ticas, Universidad
    Nacional Aut\'{o}noma de M\'{e}xico, Circuito Exterior, C.U.,
    04510 M\'{e}xico D.F., Mexico, email: \texttt{nils@ackermath.info}
  \item[Andrzej Szulkin:] Department of Mathematics, Stockholm
    University, 106 91 Stockholm, Sweden, email: \texttt{andrzejs@math.su.se}
  \end{description}
\end{sloppypar}
\end{document}